\documentclass{article}

\usepackage{graphicx}

\usepackage{authblk}

\usepackage{tikz}
\usetikzlibrary{arrows,calc}

\tikzstyle{arc}=[->, shorten <=3pt, shorten >=3pt, >=stealth, line width=1.1pt]
\tikzstyle{edge}=[shorten <=1pt, shorten >=1pt, >=stealth, line width=1.1pt]
\tikzstyle{vertex}=[circle, fill=white, draw, minimum size=5pt,
                    line width=0.75pt, inner sep=0pt, outer sep=0pt]
\tikzstyle{grayE} =[shorten <=1pt, shorten >=1pt, >=stealth, color=gray!75,
                    line width=1.1pt]
\tikzstyle{grayV}=[circle, draw, minimum size=5pt, color=gray!75,
                    line width=0.75pt, inner sep=0pt, outer sep=0pt]
\tikzstyle{blackV}=[circle, draw, fill=black, minimum size=5pt, inner sep=0pt,
                    outer sep=0pt]

\usepackage{float}

\usepackage{xcolor}

\usepackage{amsmath}
\usepackage{amssymb}

\usepackage{amsthm}

\usepackage[colorlinks=true]{hyperref}

\usepackage[capitalize]{cleveref}

\newtheorem{theorem}{Theorem}
\newtheorem{lemma}[theorem]{Lemma}
\newtheorem{corollary}[theorem]{Corollary}
\newtheorem{proposition}[theorem]{Proposition}

\theoremstyle{definition}
\newtheorem{problem}{Problem}

\usepackage{xspace}

\newcommand{\obs}{%
    \textnormal{obs}
}

\newcommand{\fgc}[1]{%
\textsc{Full} $#1$-\textsc{Colouring}\xspace
}

\title{Full homomorphisms to graph classes%
\thanks{Pavol Hell was supported by his NSERC Canada Discovery Grant. C\'esar
Hern\'andez-Cruz was supported by UNAM-PAPIIT IA106425 and SEP-CONACYT CB
A1-S-8397 grants.}}

\author[1]{Pavol~Hell\thanks{pavol@sfu.ca}}
\author[2]{C\'esar~Hern\'andez-Cruz\thanks{chc@ciencias.unam.mx}}

\affil[1]{School of Computing Science, Simon Fraser University,
          8888 University Drive, Burnaby, B.C., Canada, V5A 1S6}
\affil[2]{Facultad de Ciencias, Universidad Nacional Aut\'onoma de M\'exico,
          Av. Universidad 3000 Circuito Exterior s/n Ciudad Universitaria,
          04510, M\'exico City, M\'exico}

\begin{document}
\date{}

\maketitle
\begin{abstract}
    Given a family of graphs $\mathcal{F}$, we define a graph $G$ to be fully
    $\mathcal{F}$-colourable if $G$ admits a full homomorphism to some $F$ in
    $\mathcal{F}$.   We approach the problem of determining when a graph is
    fully $\mathcal{F}$-colourable in terms of minimal forbidden induced
    subgraphs.   We provide general results which allow to obtain the exact
    families of forbidden induced subgraphs for full $\mathcal{F}$-colouring
    when $\mathcal{F}$ is among some well-known families, such as threshold,
    trivially perfect, split, chordal, interval and strongly chordal graphs, as
    well as forests.

    Traditionally, these questions have been studied for a single graph $H$, not
    a family.   Motivated by our results on the family of forests, we contribute
    to this research by focusing on the case of a single centipede.
\end{abstract}

\section{Introduction}
\label{sec:intro}

We refer the reader to \cite{bondy2008} for standard terminology and notation.
We consider undirected graphs with neither loops nor parallel edges.   For
graphs $G$ and $H$ with disjoint sets of vertices, $G+H$ denotes their union;
the union of $n$ disjoint copies of $G$ is denoted $nG$, and for a family $G_1,
\dots, G_n$ of vertex-disjoint graphs, the union is denoted $\sum_{i=1}^n G_i$.
We use $P_n$ to denote the path graph of order $n$.

In a graph $G$ two vertices $u$ and $v$ are \textit{twins} if they have the same
closed neighbourhoods in $G$; they are \textit{false twins} if they are not
adjacent and they are \textit{true twins} otherwise.   A graph $G$ is
\textit{point-determining} if no vertex of $G$ has a false twin.

If $T$ is a tree and $u$ is a vertex of $T$, then we say that $u$ is a
\textit{support vertex} of $T$ if $u$ is not a leaf and it is adjacent to some
leaf of $T$.   Notice that trees of order less than $3$ do not have support
vertices, but every other tree has at least one support vertex.   A
\textit{caterpillar} $G$ is either $P_2$ or a tree such that the removal of all
leaves yields a path $P$.  The \textit{spine} of a caterpillar $P_2$ is $P_1$
and the spine of any other caterpillar is the path $P$. A \textit{centipede} is
a caterpillar such that every vertex in the spine has exactly one leaf attached
to it.   The \textit{$k$-centipede} is the centipede with spine $P_k$.

A family of graphs $\mathcal{F}$ is \textit{hereditary} if it is closed under
taking induced subgraphs, i.e., if $G \in \mathcal{F}$ implies $H \in
\mathcal{F}$ for every induced subgraph $H$ of $G$.   A \textit{minimal
$\mathcal{F}$-obstruction} is a graph $F$ such that $F$ is not $\mathcal{F}$,
but every proper induced subgraph of $F$ is.   Clearly, every graph that is not
in $\mathcal{F}$ contains a minimal $\mathcal{F}$-obstruction as an induced
subgraph.   Denote the set of all minimal $\mathcal{F}$-obstructions by $\obs
(\mathcal{F})$.   It follows from our previous observation that a graph is in
$\mathcal{F}$ if and only if none of its induced subgraphs is in $\obs
(\mathcal{F})$.   Therefore, if $\obs (\mathcal{F})$ is finite, it can be
checked whether a graph $G$ is in $\mathcal{F}$ in polynomial time by simply
checking each graph in $\obs (\mathcal{F})$ not to be an induced subgraph of
$G$.

Given two graphs $G$ and $H$, a \textit{full homomorphism} from $G$ to $H$, also
called a \textit{full $H$-colouring} of $G$, is a function $\varphi \colon V_G
\to V_H$ such that $uv \in E_G$ if and only if $\varphi (u) \varphi (v) \in
E_H$.   We emphasize that we use the terms full homomorphism to $H$ and full
$H$-colouring interchangeably. If a full $H$-colouring of $G$ exists, we say
that $G$ is \textit{fully $H$-colourable}.   For a fixed graph $H$, the problem
of determining whether an input graph $G$ admits a full $H$-colouring is known
as \fgc{H}.   Feder and Hell \cite{federDM308} (c.f., also \cite{ballEJC31} for
a more general result) proved that for any fixed graph $H$, the family of fully
$H$-colourable graphs has only finitely many minimal obstructions. Thus, it
follows that for each graph $H$ the problem \fgc{H} is polynomial-time solvable.

The study of concrete obstructions to full $H$-colouring, where $H$ comes from
simple families like paths and cycles, was initiated in \cite{ballEJC31} and led
to \cite{guzmanDM347}.   We continue with these investigations.

We begin by observing that if $H$ is a part of the input, i.e., is not a fixed
graph, we obtain an NP-complete problem, even for linear forests.

\begin{proposition}
    The problem of deciding for two linear forests $G$ and $H$ if there is a
    full homomorphism of $G$ to $H$ is NP-complete.    
\end{proposition}
\begin{proof}
    We proceed by a reduction from $3$-\textsc{Partition}, which is known to be
    strongly NP-complete \cite{garey1979}. Recall that an instance of
    $3$-partition is a sequence of $3n$ nonnegative integers $A = (a_1, \dots,
    a_{3n})$ whose sum is $nb$, and the question asked is whether there exists a
    partition of the given integers into $n$ disjoint groups of three such that
    each group sums exactly to $b$. For our reduction, given the sequence $A$ of
    nonnegative integers, construct the input graph $G$ as the linear forest
    with $3n$ components such that the $i$-th component has length $a_i-1$.
\end{proof}

In this paper we focus on extending the problem of \fgc{H} for one graph $H$ to
the problem of full colourability by a family $\mathcal{F}$ of graphs. Let
$\mathcal{F}$ be a fixed hereditary family of graphs; we say that a graph $G$
admits a full $\mathcal{F}$-colouring (is fully $\mathcal{F}$-colourable) if it
is fully $H$-colourable for some $H$ in $\mathcal{F}$. The problem
\fgc{\mathcal{F}} is the problem of deciding if a graph $G$ is fully
$\mathcal{F}$-colourable. So for example a graph is fully split-colourable if it
admits a full homomorphism to some split graph; in some cases we use the same
terminology for graph families that are not hereditary, e.g., a fully
tree-colourable graph is a graph that admits a full homomorphism to a tree.
Depending on the choice of $\mathcal{F}$, this problem may be polynomial-time
solvable, NP-complete or even undecidable, as we illustrate in what follows.

Recall that the
length of a string $w$ is denoted by $\ell(w)$, and define the double ended
$n$-pan to be the graph obtained by $2C_n$ by adding one bridge joining the two
copies of $C_n$.   Let $L$ be a language on $\Sigma = \{0,1\}$. For any $w \in
\Sigma^\ast$, let $S_w$ be the set of positions of $w$ which are $0$ and $T_w$
the set of positions of $w$ which are $1$. Define $G_w$ to be the disjoint union
of the double ended $(\ell(w)+6)$-pan and the graph $\sum_{n \in S_w} C_{n+6} +
\sum_{n \in T_w} \overline{C_{n+6}}$. Let $\mathcal{F}_L$ be the set
$\mathcal{F}_L = \{ G_w \colon\ w \in L \}$.

\begin{proposition}
    If $L$ is an undecidable binary language, then the full
    $\mathcal{F}_L$-colouring problem is undecidable.
\end{proposition}
\begin{proof}
    
    For an arbitrary $w \in \Sigma^\ast$, let us show that $G_w$ admits a full
    $\mathcal{F}_L$-colouring if and only if $G_w \in \mathcal{F}_L$.   For the
    non-trivial implication, suppose that $G_w$ admits a full $H$-colouring
    $\varphi$ for some $H \in \mathcal{F}_L$.   Clearly, $C_r$ admits a full
    $C_s$-colouring if and only if $r = s$, $\overline{C_r}$ admits a full
    $\overline{C_s}$-colouring if and only if $r = s$, and the double ended
    $n$-pan admits a full homomorphism to the double ended $m$-pan if and only
    if $r = s$. Also, if $r, s$, and $t$ are at least $6$, no full homomorphism
    exists between any of $C_r, \overline{C_s}$, and the double ended $t$-pan.
    Therefore, each component of $G_w$ is mapped under $\varphi$ to a copy of
    itself in $H$, and by construction of $\mathcal{F}_L$, such a copy is a
    connected component of $H$. Finally, since the double ended pan in $G_w$ is
    mapped under $\varphi$ to a copy of the same double ended pan in $H$, the
    number of connected components of $G_w$ coincides with the number of
    connected components of $H$.   Therefore, $G_w$ is isomorphic to $H$.
    
    Finally, we show that \textsc{Full $\mathcal{F}_L$-colouring} is undecidable
    whenever $L$ is undecidable.   Proceeding by contrapositive, assume that an
    algorithm $M$ to decide \textsc{Full $\mathcal{F}_L$-colouring} exists.   We
    propose an algorithm to decide $L$ using $M$ as a subroutine.   With input
    $w$, construct $G_w$, apply $M$ to $G_w$ and return the same output as $M$.
    In the previous paragraph we showed that $G_w$ admits a full
    $\mathcal{F}_L$-homomorphism if and only if $G_w \in \mathcal{F}_L$. But the
    former occurs if and only if $M$ returns \texttt{true} with input $G_w$, and
    the latter occurs if and only if $w \in L$.
\end{proof}

The rest of the article is organized as follows.   In \Cref{sec:main} we compare
the set $\obs(\mathcal{F})$ of minimal obstructions to $\mathcal{F}$ with the
set of minimal obstructions to full $\mathcal{F}$-colouring. This allows us to
conclude in our main result in this section a description of the sets of minimal
obstructions to full $\mathcal{F}$-colouring for the well-known families
$\mathcal{F}$ of chordal graphs, strongly chordal graphs, interval graphs, split
graphs, threshold graphs, quasi-threshold graphs and forests. In \Cref{sec:algo}
we use the characterization of fully threshold-colourable graphs obtained in the
previous section to propose a certifying algorithm to recognize this family of
graphs in time $O(|V|+|E|)$ using elementary data structures. We analyze the
structure of the minimal obstructions to full $H$-colouring when $H$ is a single
centipede and derive the minimal obstructions to the family of all centipedes in
\Cref{sec:centipedes}.   Possible future research directions are presented in
\Cref{sec:conclusions}.

\section{Minimal obstructions to full
\texorpdfstring{$\mathcal{F}$}{F}-colouring}
\label{sec:main}

Let $\mathcal{F}$ be a hereditary class of graphs.  Clearly, every graph in
$\mathcal{F}$ is fully $\mathcal{F}$-colourable via the identity mapping. Since
the converse of the previous statement is not true, it is natural to ask how the
complexity of \fgc{\mathcal{F}} is related to the complexity of the recognition
problem for $\mathcal{F}$; let us call the latter
$\mathcal{F}$-\textsc{Recognition}. There is a simple polynomial reduction from
\fgc{\mathcal{F}} to $\mathcal{F}$-\textsc{Recognition}.  If $G'$ is obtained
from an input graph $G$ for \fgc{\mathcal{F}} by removing all false twins, then
clearly $G$ is fully $\mathcal{F}$-colourable if and only if $G'$ is in
$\mathcal{F}$.   The opposite direction yields an interesting problem which we
state formally in \Cref{sec:conclusions}.

Let us consider a simple example.   When $\mathcal{K}$ is the family of complete
graphs we have $\obs(\mathcal{K}) = \{ 2K_1 \}$.   If $G$ admits a full
homomorphism to $K_r$, then $G$ is a complete $r$-partite graph.   Therefore,
the set of minimal obstructions to full $\mathcal{K}$-colouring is $\{ K_1 + K_2
\}$.   In this example, the unique graph in $\obs(\mathcal{K})$ is not a minimal
obstruction to full $\mathcal{K}$-colouring, but it is an induced subgraph of
the unique minimal obstruction to full $\mathcal{K}$-colouring.   It is natural
to ponder two questions: What caused $2K_1$ not to be a minimal obstruction to
full $\mathcal{K}$-colouring?  Is it true for any hereditary family
$\mathcal{F}$ that a graph in $\obs(\mathcal{F})$ which is not a minimal
obstruction to full $\mathcal{F}$-colouring must be an induced subgraph of such
an obstruction?

The following simple facts answer the questions posed in the previous paragraph.
The proofs are easy and thus omitted.

\begin{proposition}
\label{pro:fullFObs}
    Let $\mathcal{F}$ be a hereditary family of graphs.
    \begin{enumerate}
        \item Every point-determining obstruction to $\mathcal{F}$ is also an
            obstruction to full $\mathcal{F}$-colouring.
        
        \item Any minimal obstruction to full $\mathcal{F}$-colouring contains a
            graph in $\obs(\mathcal{F})$ as an induced subgraph.
    
        \item If $G$ is a fully $\mathcal{F}$-colourable graph, then each of its
            point-determining induced subgraphs is in $\mathcal{F}$.
    \end{enumerate}
    \hfill $\square$
\end{proposition}

Given a hereditary class of graphs $\mathcal{F}$, there are some properties of
the graphs in $\mathcal{F}$ that can be deducted from the minimal $\mathcal{F}$
obstructions. For example, $\mathcal{F}$ is additive if and only if all graphs
in $\obs({\mathcal{F}})$ are connected, or $\mathcal{F}$ is closed under taking
complements if and only if $\obs(\mathcal{F})$ is.   Our next result follows
this line of thought, and relates it to full homomorphisms.

\begin{proposition}
\label{pro:point-det}
    Let $\mathcal{F}$ be a hereditary property of graphs.   The following
    statements are equivalent.
    \begin{enumerate}
        \item Every graph in $\obs(\mathcal{F})$ is point-determining.
        \item $\mathcal{F}$ is closed under vertex duplication.
        \item A graph $G$ is fully $\mathcal{F}$-colourable if and only if it is
            in $\mathcal{F}$.
    \end{enumerate}
\end{proposition}
\begin{proof}
    Suppose that every graph in $\obs(\mathcal{F})$ is point-determining.   If
    $G$ is a graph in $\mathcal{F}$, then duplicating any vertex in $G$ cannot
    create a subgraph of $G$ which is in $\obs(\mathcal{F})$, so the resulting
    graph is also in $\mathcal{F}$.

    To prove that the second item implies the third one, suppose that
    $\mathcal{F}$ is closed under vertex duplication and let $G$ be a fully
    $H$-colourable graph for some $H$ in $\mathcal{F}$.   If the full
    homomorphism is injective, then $G$ is isomorphic to $H$, and hence, $G$ is
    in $\mathcal{F}$.   Otherwise, all vertices in each colour class of the full
    $H$-colouring of $G$ are false twins, so $G$ is obtained from $H$ by
    duplicating vertices.   As $\mathcal{F}$ is closed under vertex duplication,
    we conclude that $G$ is in $\mathcal{F}$.   The remaining implication of the
    third item is straightforward to verify.

    To prove that the last item implies the first one, we proceed by
    contrapositive.   Let $F$ be a graph in $\obs(\mathcal{F})$ which is not
    point-determining, and let $u$ and $v$ be a pair of false twins in $F$.
    Clearly, $F-u$ is in $\mathcal{F}$, and because $F$ is fully
    $(F-u)$-colourable, we conclude that $F$ is fully $\mathcal{F}$-colourable.
    We have found a graph which is not in $\mathcal{F}$ but admits a full
    homomorphism to a graph in $\mathcal{F}$, which concludes the proof.
\end{proof}

It is worth noting that for many classic hereditary classes the first condition
of \Cref{pro:point-det} is easily verifiable to be true.   Recall that a graph
is a \textit{bipartite chain graph} if and only if it is $\{ C_3, 2K_2, C_5
\}$-free \cite{mahadev1995}.   It follows from the Strong Perfect Graph Theorem
\cite{chudnovskyAM164} that a graph is perfect if and only if it is free from
holes and odd-holes.   Possibly the best known characterization for cographs
states that the only minimal obstruction for the class of cographs is $P_4$
\cite{corneilDAM3}.  It is easy to verify that odd holes, odd antiholes, $C_3$,
$2K_2$ and $P_4$ are point-determining.   A tedious but straightforward
exploration yields a similar observation for the minimal obstructions to the
classes of permutation graphs, comparability graphs and distance hereditary
graphs.   Thus, \Cref{pro:point-det} implies the following result.

\begin{proposition}
    For the following families $\mathcal{F}$, the set of minimal obstructions to
    full $\mathcal{F}$-colouring is the same as the set of minimal obstructions
    for the family $\mathcal{F}$, i.e., $\obs(\mathcal{F})$.
    \begin{enumerate}
        \item Bipartite chain graphs.
        \item Cographs.
        \item Distance hereditary graphs.
        \item Permutation graphs.
        \item Comparability graphs.
        \item Perfect graphs.
    \end{enumerate}
\end{proposition}

So, if we have a good understanding of $\mathcal{F}$, then \textsc{Full
$\mathcal{F}$-colouring} is an interesting problem only when $\obs(\mathcal{F})$
contains graphs that are not point-determining.   A case having a number of
appealing examples is when the only non-point-determining graph in
$\obs(\mathcal{F})$ is $C_4$.   Let $F_{C_4}$ be the set containing the graphs
depicted in \Cref{fig:FC4}, this is,
\[
    F_{C_4} = \{A, B, \textnormal{house}, \overline{X_{170}}, \overline{2P_3}\}.
\]

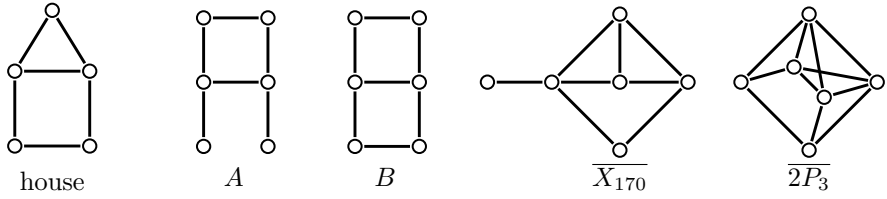
\begin{figure}[ht!]
    \centering
    \begin{tikzpicture}
    
    \begin{scope}[xshift=0cm, yshift=-0.35cm]
    \foreach \i in {0,...,3}
        \node [vertex] (\i) at ({360/4*\i + 45}:0.7){};
    \node [vertex] (r) at (0,1.3){};
    
    \foreach \i in {0,...,3}
        \draw let \n1={int(mod(\i+1,4))} in [edge] (\i) to (\n1);
    
    \draw [edge] (r) to (0);
    \draw [edge] (r) to (1);

    \node [rectangle] (n) at (0,-0.95){house};
    \end{scope}
    
    \begin{scope}[xshift=2cm]
    \foreach \i in {-1,0,1}{
        \node [vertex] (\i-0) at (0,{0.85*\i}){};
        \node [vertex] (\i-1) at (0.85,{0.85*\i}){};
    }
    
    \foreach \i in {0,1}{
        \draw [edge] (\i-0) to (\i-1);
        \draw let \n1={int(\i-1)} in [edge] (\i-0) to (\n1-0);
        \draw let \n1={int(\i-1)} in [edge] (\i-1) to (\n1-1);
    }

    \node [rectangle] (n) at (0.4,-1.25){$A$};
    \end{scope}

    \begin{scope}[xshift=4cm, yshift=0cm]
    \foreach \i in {-1,0,1}{
        \node [vertex] (\i-0) at (0,{0.85*\i}){};
        \node [vertex] (\i-1) at (0.85,{0.85*\i}){};
    }
    
    \foreach \i in {0,1}{
        \draw [edge] (\i-0) to (\i-1);
        \draw let \n1={int(\i-1)} in [edge] (\i-0) to (\n1-0);
        \draw let \n1={int(\i-1)} in [edge] (\i-1) to (\n1-1);
    }

    \draw [edge] (-1-0) to (-1-1);

    \node [rectangle] (n) at (0.4,-1.25){$B$};
    \end{scope}

    \begin{scope}[xshift=7.5cm]
        \foreach \i in {0,...,3}
            \node [vertex] (\i) at ({360/4*\i}:0.9){};
        \node [vertex] (in) at (0,0){};
        \node [vertex] (out) at (-1.75,0){};
        
        \foreach \i in {0,...,3}
            \draw let \n1={int(mod(\i+1,4))} in [edge] (\i) to (\n1);
        \foreach \i in {0,1,2}
            \draw [edge] (\i) to (in);
        \draw [edge] (2) to (out);
        
        \node [rectangle] (n) at (0,-1.25){$\overline{X_{170}}$};
    \end{scope}

    \begin{scope}[xshift=10cm]
        \foreach \i in {0,...,3}
            \node [vertex] (\i) at ({360/4*\i}:0.9){};
        \node [vertex] (l) at (-0.2,0.2){};
        \node [vertex] (r) at (0.2,-0.2){};
        
        \foreach \i in {0,...,3}
            \draw let \n1={int(mod(\i+1,4))} in [edge] (\i) to (\n1);
        \foreach \i in {0,1,3,l}
            \draw [edge] (\i) to (r);
        \foreach \i in {0,1,2}
            \draw [edge] (\i) to (l);

        \node [rectangle] (n) at (0,-1.25){$\overline{2P_3}$};
    \end{scope}
    
    \end{tikzpicture}
    \caption{The family $F_{C_4}$.}
    \label{fig:FC4}
\end{figure}

\begin{lemma}
\label{lem:C4}
    Let $\mathcal{F}$ be a hereditary property.   If $C_4$ is the only graph in
    $\obs(\mathcal{F})$ which is not point-determining, then the set of minimal
    obstructions to full $\mathcal{F}$-colouring is obtained from the union of
    $\obs(\mathcal{F}) \setminus \{ C_4 \}$ and $F_{C_4}$ by removing graphs
    from the latter which are not $(\obs(\mathcal{F}) \setminus \{ C_4
    \})$-free.
\end{lemma}
\begin{proof}
    It is easy to observe that graphs in $\obs(\mathcal{F}) \setminus \{C_4\}$
    are minimal obstructions to full $\mathcal{F}$-colouring.   Every graph in
    $\mathcal{F}$ is fully $\mathcal{F}$-colourable, so minimal obstructions to
    full $\mathcal{F}$-colouring not in $\obs(\mathcal{F})$ must contain an
    induced copy of $C_4$.   Since $C_4$ is in $\obs(\mathcal{F})$, in any full
    $\mathcal{F}$-colouring of a graph containing an induced copy of $C_4$, such
    a copy must be mapped to a $P_2$ or a $P_3$ in the target graph.   From
    here, and because of their point-determination, it is routine to verify that
    none of the graphs in $F_{C_4}$ is fully $\mathcal{F}$-colourable.
    Moreover, also because of their point-determination, it is clear that unless
    they properly contain a graph in $\obs(\mathcal{F})$ as an induced subgraph,
    they are minimal obstructions to full $\mathcal{F}$-colouring.   So, let us
    verify that these are the only possibilities to extend our family of minimal
    obstructions to full $\mathcal{F}$-colouring.
    
    Let $F$ be a minimal obstruction to full $\mathcal{F}$-colouring having $C$
    as an induced copy of $C_4$ with $C = (u_1, u_2, u_3, u_4, u_1)$.  As we
    observed earlier, minimal obstructions to full $\mathcal{F}$-colouring must
    be point-determining, and hence, we assume without loss of generality that
    $u_1$ has a neighbour $v$ which is not a neighbour of $u_3$, and $u_2$ has a
    neighbour $w$ which is not a neighbour of $u_4$.   If $v = w$, then $F$ is
    isomorphic to the house.   Else, call $N_C(v)$ and $N_C(u)$ the intersection
    of $V_C$ with the neighbourhoods of $v$ and $u$, respectively.  We have that
    $N_C(v)$ is either $\{ u_1 \}$ or $\{ u_1, u_2, u_4 \}$, and $N_C(w)$ is
    either $\{ u_2 \}$ or $\{ u_1, u_2, u_3 \}$.

    If $N_C(v) = \{ u_1 \}$ and $N_C(w) = \{ u_2 \}$,  then we have that either
    $v$ is adjacent to $w$, and hence $F$ is isomporphic to $B$, or $v$ is not
    adjacent to $w$, and thus $F$ is isomorphic to $A$.   If $N_C(v) = \{ u_1
    \}$ and $N_C(w) = \{ u_1, u_2, u_3 \}$, then $v$ is not adjacent to $u$ (as
    otherwise $F$ would properly contain the house, see \Cref{fig:split-proof}),
    and thus $F$ is isomorphic to $\overline{X_{170}}$. An analogous argument
    applies if $N_C(v) = \{ u_1, u_2, u_4 \}$ and $N_C(w) = \{ u_2 \}$.   The
    last case is when $N_C(v) = \{ u_1, u_2, u_4 \}$ and $N_C(w) = \{ u_1, u_2,
    u_3 \}$.   Notice that $u$ is adjacent to $w$, as otherwise $F$ would
    properly contain the house (see \Cref{fig:split-proof}). Thus, $F$ is
    isomorphic to $\overline{2P_3}$.
\end{proof}

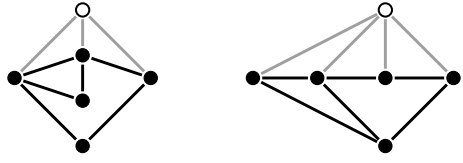
\begin{figure}[ht!]
    \centering
    \begin{tikzpicture}

    \begin{scope}[xshift=3cm]
        \node [blackV] (0) at (0:0.9){};
        \node [vertex] (1) at (90:0.9){};
        \node [blackV] (2) at (180:0.9){};
        \node [blackV] (3) at (270:0.9){};
        \node [blackV] (d) at (0,-0.3){};
        \node [blackV] (u) at (0,0.3){};
        
        \foreach \i in {2,3}
            \draw let \n1={int(mod(\i+1,4))} in [edge] (\i) to (\n1);
        \foreach \i in {0,2,d}
            \draw [edge] (\i) to (u);
        \foreach \i in {0,2,u}
            \draw [grayE] (\i) to (1);
        \draw [edge] (2) to (d);
    \end{scope}
    
    \begin{scope}[xshift=7cm]
        \node [blackV] (0) at (0:0.9){};
        \node [vertex] (1) at (90:0.9){};
        \node [blackV] (2) at (180:0.9){};
        \node [blackV] (3) at (270:0.9){};
        \node [blackV] (in) at (0,0){};
        \node [blackV] (out) at (-1.75,0){};
        
        \foreach \i in {2,3}
            \draw let \n1={int(mod(\i+1,4))} in [edge] (\i) to (\n1);
        \foreach \i in {0,2}
            \draw [edge] (\i) to (in);
        \foreach \i in {2,3}
            \draw [edge] (\i) to (out);
        \foreach \i in {in, out, 0, 2}
            \draw [grayE] (1) to (\i);
    \end{scope}
    
    \end{tikzpicture}
    \caption{Non-minimal cases in the proof of \Cref{lem:C4}.}
    \label{fig:split-proof}
\end{figure}

In the spirit of \Cref{lem:C4}, we now consider other graphs that are not
point-determining as additional minimal obstructions. Observe that false twins
in a forest are either a pair of isolated vertices or leaves adjacent to the
same support vertex. For a tree $T$ which is not point-determining, we define a
point-determining tree $T^p$ which is obtained from $T$ by considering each set
of twin vertices $u_1, \dots, u_k$ with common neighbour $u$ and subdividing the
edges $uu_i$ for each $i \in \{u_1, \dots, u_{k-1}\}$.   For example,
$K_{1,3}^p$ is the graph $E$.   Let $\mathcal{F}$ be a hereditary family of
forests, and let $\obs(\mathcal{F})$ be the family of its minimal obstructions.
Define $F_\mathcal{F}$ to be the family obtained from $\obs(\mathcal{F})$ by
substituting each non-point-determining component $T$ by $T^p$; also include in
$F_{\mathcal{F}}$ the graphs $A$ and $B$ and every cycle of length different
than $4$.   If graphs in $\obs(\mathcal{F})$ have at most one isolated vertex,
then each element of $F_{\mathcal{F}}$ is point-determining and contains a
member of $\obs(\mathcal{F})$ as an induced subgraph. Thus, by
\Cref{pro:fullFObs,pro:point-det}, if $G$ admits a full $\mathcal{F}$-colouring,
then it is $F_\mathcal{F}$-free. This is in fact an equivalence.

\begin{proposition}
\label{pro:tree-obstructions}
    Let $\mathcal{F}$ be a hereditary family of forests such that graphs in
    $\obs(\mathcal{F})$ have at most one isolated vertex.   A graph $G$ admits a
    full $\mathcal{F}$-colouring if and only if it is $F_\mathcal{F}$-free.
\end{proposition}
\begin{proof}
    We proceed to prove the missing direction by contrapositive.   Suppose that
    $G$ is not fully $\mathcal{F}$-colourable and let $G'$ be the
    point-determining graph obtained from $G$ by removing all false twins.
    Clearly, $G'$ is not in $\mathcal{F}$, as otherwise $G$ would be fully
    $\mathcal{F}$-colourable. Hence, $G'$ contains a graph $F$ in
    $\obs(\mathcal{F})$ as a minimal obstruction.   If $F$ is point-determining,
    then it belongs to $F_{\mathcal{F}}$.  Since $A, B$ and every cycle of
    length different from $4$ are elements of $F_{\mathcal{F}}$, if $G'$ is not
    a forest, then $G$ is not $F_{\mathcal{F}}$-free.   So, suppose that $F$ is
    not point-determining. By hypothesis, pairs of false twins in $F$ are leaves
    with a common support vertex.  Because $G'$ is a point-determining forest,
    for each set $\{ u_1, \dots, u_k \}$ of pairwise false twins in $F$ with
    common support vertex $u$ it must be the case that, maybe except for one,
    each of them has a neighbour in $G'$ which is not a neighbour of any of the
    rest.  As a result, $G'$ contains an induced copy of $F$ with each edge
    $uu_1, \dots, uu_{k-1}$ subdivided.   In other words, $G'$ contains a graph
    in $F_\mathcal{F}$ as an induced subgraph. Therefore, in all cases $G$
    contains a graph in $F_\mathcal{F}$ as an induced subgraph.
\end{proof}

The ideas behind \Cref{lem:C4} and \Cref{pro:tree-obstructions} are similar: we
find the needed minimal obstructions by extending those graphs in
$\obs(\mathcal{F})$ which are not point-determining to a set of
point-determining supergraphs. Similar ideas lead to Lemma 5.1 in
\cite{bodirskyArXiv}\footnote{Although \cite{bodirskyArXiv} is the arXiv version
of \cite{bodirskyLIPI380}, the relevant result is included in an appendix not
appearing in \cite{bodirskyLIPI380}.}. These developments were independent of
each other; the results in \cite{bodirskyArXiv} are more general both in scope
and in context while our results are more specific in describing the concrete
extensions. In \cite{bodirskyArXiv} the results are stated for general
relational structures; in the context of graphs they amount to bounding the size
of minimal obstructions to full $\mathcal{F}$-colourability as a function of the
maximum size of a graph in $\obs(\mathcal{F})$, for any hereditary family
$\mathcal{F}$. This is an extension of the result of \cite{federDM308} mentioned
in the introduction, and offers another proof of a result in \cite{ballEJC31}.

The strategy described in the previous paragraph may be used to extend the
result of \Cref{pro:tree-obstructions} to the case of an arbitrary hereditary
family of forests $\mathcal{F}$, at the cost of losing knowledge of the graphs
in $F_{\mathcal{F}}$.   We have observed that there is a unique way of extending
a tree $T$ to a point-determining tree $T^p$---the situation in a forest is
quite different.   Let $S$ be the set of isolated vertices in a forest $G$.   If
$S$ has more than one element, then it consists of pairwise false twins. For a
point-determining forest $H$ with independence number $|S|$, let $G^p$ be the
forest obtained from $G$ by replacing each non-point-determining component $T$
of $G$ by $T^p$ and the set $S$ by $H$. As we have discussed, $G^p_H$ is a
point-determining forest that contains $G$ as an induced subgraph.   Hence, we
can extend the construction of $F_{\mathcal{F}}$ with this modification, i.e.,
substituting $S$ by every possible point-determining forest with independence
number $|S|$, to obtain a result analogous to \Cref{pro:tree-obstructions}.
Nonetheless, we no longer obtain a single forest in $F_{\mathcal{F}}$ for each
non-point-determining forest in $\obs(\mathcal{F})$. Notice that knowing the
complete family of point-determining forests with independence number $k$ allows
us to describe this extended version of $F_{\mathcal{F}}$ precisely. Therefore,
we propose the problem of enumerating, for every positive integer $k$, the
family of point-determining forests with independence number exactly $k$.   To
finish this discussion, let us stress that the obstructions to full
$\mathcal{F}$-colouring in $F_{\mathcal{F}}$ are not necessarily minimal.   As
an obvious example, if $\mathcal{F}$ contains a linear forest, then
$F_{\mathcal{F}}$ also contains a linear forest, and hence, it is no longer
necessary for $F_{\mathcal{F}}$ to contain all cycles of length different from
$4$.  

An immediate consequence of \Cref{lem:C4} is that a graph admits a full
homomorphism to a $C_4$-free graph if and only if it is $F_{C_4}$-free.   Below
we present some other interesting direct consequences of \Cref{lem:C4}.

Lekkerkerker and Boland identified the family of minimal forbidden induced
subgraphs for a graph to be an interval graph \cite{lekkerkerkerFM51}, they are
all the cycles of length greater than $3$ together with the two individual
graphs and three infinite families depicted in \Cref{fig:lekkerkerker}; call
$\mathcal{I}$ the set of graphs depicted in \Cref{fig:lekkerkerker}.   A
\textit{trampoline} is a graph obtained from a complete graph $K$ of order at
least $3$ by choosing a hamiltonian cycle $C$ in $K$ and adding a new vertex for
each edge $e$ of $C$ and making it adjacent to both endpoints of $e$; call $T_n$
the trampoline obtained from $K_n$.  In \cite{farberDM43} Farber characterized
strongly chordal graphs as the trampoline-free chordal graphs.  In the theorem
below we also use the notation from \Cref{fig:FC4} for house
$\overline{X_{170}}$ and $\overline{2P_3}$.

\begin{figure}[htb!]
\centering
\begin{tikzpicture}
    \begin{scope}
        \node [vertex] (c) at (0,0){};
        \foreach \i in {0,1,2}{
            \node [vertex] (\i) at ({90+360/3*\i}:0.75){};
            \node [vertex] (\i-1) at ({90+360/3*\i}:1.5){};
        }
        \foreach \i in {0,1,2}{
            \draw [edge] (c) to (\i);
            \draw [edge] (\i) to (\i-1);
        }
    \end{scope}
    \begin{scope}[xshift=2.8cm,yshift=0.45cm]
        \node [vertex] (c) at (0,0){};
        \node [vertex] (u) at (0,1){};
        \node [vertex] (d) at (0,-1){};
        \node [vertex] (l) at (-1,0){};
        \node [vertex] (r) at (1,0){};
        \node [vertex] (ld) at (-1,-1){};
        \node [vertex] (rd) at (1,-1){};
        \foreach \v in {u,l,d,r}
            \draw [edge] (c) to (\v);
        \foreach \v in {ld,l,r,rd}
            \draw [edge] (d) to (\v);
        \foreach \u/\v in {ld/l,rd/r}
            \draw [edge] (\u) to (\v);
    \end{scope}
    \begin{scope}[xshift=5.7cm]
        \node [vertex] (m) at (0,-0.375){};
        \foreach \i in {0,1,2}{
            \node [vertex] (\i) at ({90+360/3*\i}:0.75){};
            \node [vertex] (\i-1) at ({90+360/3*\i}:1.5){};
        }
        \foreach \i in {0,1,2}
            \draw [edge] (\i) to (\i-1);
        \foreach \u/\v in {0/1,2/0,1/m,0/m}
            \draw [edge] (\u) to (\v);
        \draw [edge,dashed] (m) to (2);
    \end{scope}
    \begin{scope}[xshift=9cm]
        \node [vertex] (a) at (240:0.865){};
        \node [vertex] (b) at (300:0.865){};
        \foreach \i in {1,2}
            \node [vertex] (\i) at ({270+360/3*\i}:0.75){};
        \foreach \i in {0,1,2}
            \node [vertex] (\i-1) at ({90+360/3*\i}:1.5){};
        \foreach \u/\v in {0-1/1,0-1/2,1-1/2,2-1/b,2-1/1,1/2}
            \draw [edge] (\u) to (\v);
        \foreach \u in {1,2}
            \foreach \v in {a,b}
                \draw [edge] (\u) to (\v);
        \draw [edge] (1-1) to (a);
        \draw [edge] (2-1) to (b);
        \draw [edge,dashed] (a) to (b);
    \end{scope}
\end{tikzpicture}
\caption{Family $\mathcal{I}$ of Lekkerkerker and Boland's minimal
obstructions for interval graphs (other than cycles). Dashed edges represent
(possibly trivial) paths.}
\label{fig:lekkerkerker}
\end{figure}
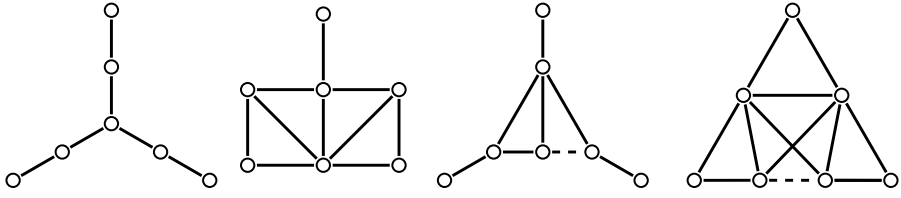

\begin{theorem}
\label{thm:C4-consequences}
    For the following families, we exhibit the set $\mathcal{S}$ of minimal
    obstructions to full colouring.
    \begin{enumerate}
        \item For quasi-threshold graphs, $\mathcal{S} = \{ P_4,
        \overline{2P_3}\}$.

        \item For threshold graphs, $\mathcal{S} = \{ 2K_2, P_4,
        \overline{2P_3}\}$.

        \item For split graphs, $\mathcal{S} = \{ 2K_2, C_5, A,
            \textnormal{house}, \overline{X_{170}}, \overline{2P_3} \}$.

        \item For forests, $\mathcal{S} = \{ A, B \} \cup \{ C_n \colon\ n = 3
            \lor n \ge 5 \}$. \label{itm:forest}
            
        \item For chordal graphs, $\mathcal{S} = F_{C_4} \cup \{ C_n \colon\ n
        \ge 5 \}$.
            
        \item For interval graphs, $\mathcal{S} = F_{C_4} \cup \{ C_n \colon\ n
            \ge 5 \} \cup \mathcal{I}$.
            
        \item For strongly chordal graphs, $\mathcal{S} = F_{C_4} \cup \{ C_n
            \colon\ n \ge 5 \} \cup \{ T_n \colon\ n \ge 3 \}$.
    \end{enumerate}
\end{theorem}
\begin{proof}
    In all cases, the only graph in $\obs(\mathcal{F})$ which is not
    point-determining is $C_4$, so in order to obtain the desired set of minimal
    obstructions by applying \Cref{lem:C4}, it suffices to verify which graphs
    in $F_{C_4}$ contain some graph in $\obs(\mathcal{F})$ as an induced
    subgraph.   To this end, notice that no graph in $F_{C_4}$ contain a cycle
    of length greater than $4$ as an induced subgraph.

    For the first two items notice that $P_4$ is an induced subgraph of the
    graphs $\textnormal{house}, A, B$ and $\overline{X_{170}}$.   Therefore, in
    both cases $\overline{2P_3}$ is the only graph in $F_{C_4}$ which ends up in
    the set of minimal obstructions to full $\mathcal{F}$-colouring.

    For the third item, notice that from within $F_{C_4}$, only $B$ contains
    $2K_2$.

    For the case of forests, observe that $A$ and $B$ are the only graphs in
    $F_{C_4}$ which do not contain a triangle.

    For the last three items, it is clear that the only graph in
    $\obs(\mathcal{F})$ contained in a graph from $F_{C_4}$ is $C_4$.
\end{proof}

Recall that a graph is distance-hereditary if and only if it is hole-free and
$\{B, \textnormal{gem}, \textnormal{house} \}$-free. \Cref{itm:forest} of
\Cref{thm:C4-consequences} characterizes fully forest-colourable graphs in terms
of minimal obstructions.  From this characterization it is trivial to verify
that the class of fully forest-colourable graphs is contained in the class of
distance-hereditary graphs.  As our following result shows, we can also describe
the family of fully forest-colourable graphs in terms of a restricted version of
the recursive construction that characterizes distance-hereditary graphs.   We
deal with the connected case first.

\begin{corollary}
\label{cor:tree-blowup}
    Let $G$ be a connected chordal bipartite graph.   The following statements
    are equivalent.
    \begin{enumerate}
        \item $G$ is $\{A,B\}$-free.
        
        \item There exists a tree $T$ such that $G$ is fully $T$-colourable.
        
        \item $G$ can be constructed from $K_1$ by adding pendant vertices and
            then creating false twins. (All additions of pendant vertices happen
            before all the creations of false twins.)
    \end{enumerate} 
\end{corollary}

\begin{proof}
    The equivalence of the first two items comes from
    \Cref{thm:C4-consequences}.   We now prove that the first item implies the
    third one.

    Suppose first that $G$ is $\{A,B\}$-free.   Let $H$ be a maximal point
    determining induced sugbraph of $G$.   We affirm that $H$ is a tree.   Since
    $G$ is chordal bipartite, any cycle in $G$ is a $4$-cycle.   Aiming for a
    contradiction, suppose that $(a,b,c,d,a)$ is a $4$-cycle in $H$.   Since $H$
    is point-determining, we may assume without loss of generality that there
    are vertices $x$ and $y$ such that $x$ is a neighbour of $a$ and but not of
    $c$, and $y$ is a neighbour of $b$ but not of $d$.   But now, $\{a, b, c, d,
    x, y\}$ induces an $A$ graph if $x$ is not adjacent to $y$, and a $B$ graph
    otherwise.   The contradiction arises from assuming the existence of a cycle
    in $H$, so it is acyclic.   The fact that $H$ is connected comes from its
    maximality and the connectedness of $G$.   Notice that the maximality of $H$
    implies that every vertex in $G$ is either in $H$ or is a false twin of a
    vertex in $H$.    From here, it is clear that $G$ can be obtained from $T$
    by creating false twins.   Therefore, $G$ can be obtained from $K_1$ by
    adding pendant vertices until $T$ is constructed, and then creating false
    twins.

    We finish the proof by verifying that the third item implies the second one.
    Let $T$ be the tree obtained after all the pendant vertex additions have
    been applied. Since $G$ is created by adding false twins to vertices in $T$,
    an obvious full $T$-colouring of $G$ is obtained by mapping each new vertex
    to its false twin originally in $T$.
\end{proof}

If $F$ is a forest and $T$ is any tree obtained from $F$ by adding a new vertex
$v$ and making it adjacent to each component of $F$ with a single edge, then it
is clear that any graph $G$ admitting a full $F$-colouring also admits a full
$T$-colouring.   Thus, any fully forest-colourable graph is also a fully
tree-colourable graph.   Trivially, any fully tree-colourable graph is also a
fully forest-colourable graph, so when $G$ is not necessarily connected, the
first two conditions in \cref{cor:tree-blowup} are also equivalent to fully
forest-colourability.   The recursive construction in the third item of
\Cref{cor:tree-blowup} always creates connected graphs.  Nonetheless, a simple
modification in the construction allows us to construct any graph admiting a
full forest-colouring.   It suffices to add a first step in the recursive
construction where we create false twins starting from $K_1$; after this step,
we end up with an independent set, and each of its vertices is then expanded
into a fully tree-colourable graph.

As we have mentioned, Guzm\'an-Pro proved in \cite{guzmanDM347} that the family
obtained by forbidding all cycles of length different from $4$, together with
graphs $A, B$ and $E$, is precisely the family of fully linear-forest-colourable
graphs. Recalling that claw-free forests are precisely linear forests, it is
natural to ask what happens when we remove the graph $E$ from the set of
forbidden induced subgraphs that characterize fully linear-forest-colourable
graphs.  \Cref{itm:forest} in \Cref{thm:C4-consequences} provides a nice answer
to this question, stating that such a set of forbidden induced subgraphs
characterizes the class of fully forest-colourable graphs.
\cref{cor:tree-blowup} provides an analogous answer for connected fully
tree-colourable graphs.

We finish this section by exploring the case when we forbid the graphs $A, B$,
and $E$, but unlike \cite{guzmanDM347}, allow cycles of any length.   We need to
forbid some additional graphs to obtain an interesting class.   In the proof of
our next result we use the following classic tool.   Consider a graph $G$ and a
subset $X$ of $V$. For a subset $S$ of $X$ the \textit{exact neighbourhood of
$S$ with respect to $X$} is the subset $N_{S,X}$ consisting of vertices $v$ in
$V \setminus X$ such that $N(v) \cap X = S$. When $X$ is clear from the context
we use $N_S$ instead of $N_{S,X}$. The \textit{complete partition of $V$ induced
by $X$}, denoted $\mathcal{P}_X$, is defined to be the set
\[
	\mathcal{P}_X = X \cup \{ N_S \colon\ S \subseteq X \},
\]
which is clearly a partition of $V$ (some of the parts of $\mathcal{P}_X$ may be
empty).   If $S$ is a small set, e.g., $S = \{ a, b \}$, we use $N_{ab}$ instead
of $N_{\{a,b\}}$, or if $S = \{ v_i, v_j \}$ we use $N_{ij}$.

\begin{theorem}
\label{thm:deg2-blowup}
    A graph $G$ admits a full homomorphism to a graph of maximum degree at most
    $2$ if and only if it is $\{ \textnormal{paw}, K_4, 5\textnormal{-pan}, A,
    B, E \}$-free.
\end{theorem}

\begin{proof}
    Suppose without loss of generality that $G$ is connected.   If every induced
    cycle of $G$ has length $4$, then $G$ is fully path-colourable, so the
    result holds.    Suppose that $G$ contains a triangle with vertex set $X =
    \{ v_0, v_1, v_2 \}$.   Consider the complete partition induced by $X$, and
    notice that, since $G$ is paw-free, we have $N_i = \varnothing$ for each $i
    \in \{ 0, 1, 2 \}$.  Clearly, not having $K_4$ as a subgraph of $G$ implies
    that $N_{012} = \varnothing$.   Recall that we assumed $G$ connected, and
    notice that $X$ must dominate every vertex in $V \setminus X$, as otherwise,
    by considering a shortest path from a vertex not dominated by $X$ to $X$, we
    would find a paw as an induced subgraph of $G$.   Thus, $N_\varnothing =
    \varnothing$.   Finally, recalling that $G$ is paw-free, it is easy to
    obtain that $N_{ij}$ is completely adjacent to $N_{jk}$ for $\{ i, j, k \} =
    \{ 0, 1, 2 \}$.   Therefore, $(\{ v_0 \} \cup N_{12}, \{ v_1 \} \cup N_{02},
    \{ v_2 \} \cup N_{01})$ is a complete $3$-partition of $G$.

    Now, suppose that, except for the existence of $4$-cycles, the shortest
    cycle in $G$ has length $k$, with $k \ge 5$.   Again, assuming that the
    vertex set of such a cycle is $X = \{ v_0, \dots, v_{k-1} \}$, we consider
    the complete partition $\mathcal{F}_X$ induced by $X$.   The absence of the
    $5$-pan and the $E$ graph as induced subgraphs forces $N_i$ to be empty for
    every $i \in \{ 0, \dots, k-1 \}$.   For any $S \subset X$, if $|S| \ge 3$,
    then $N_S$ must be empty, as otherwise $G$ would contain a cycle of length
    strictly less than $k$ and different from $4$.   For the same reason, when
    $|S| = 2$, we have that $N_S \ne \varnothing$ if and only if $S = \{ v_i,
    v_{i+2} \}$ for some $i \in \{ 0, \dots, k-1 \}$, with subscripts taken
    modulo $k$.   But now, as $G$ is connected and $A$-free, also
    $N_\varnothing$ is empty, because the only possibility for a vertex in
    $N_\varnothing$ to reach $C_k$ would be through a vertex in $N_{i(i+2)}$,
    which would create an induced $A$.   Finally, since $G$ is $B$-free, $N_{i
    (i+2)}$ is completely adjacent to $N_{(i+1)(i+3)}$ for each $i \in \{ 0,
    \dots, k-1 \}$ with subscripts taken modulo $k$.   Therefore, $(\{ v_0 \}
    \cup N_{(k-1)2}, \{ v_1 \} \cup N_{02}, \dots, \{ v_{k-1} \} \cup
    N_{(k-2)0})$ is a full $C_k$-homomorphism.
\end{proof}

\section{Recognition of full threshold-colourability}
\label{sec:algo}

\Cref{thm:C4-consequences} suggests the possible existence of efficient
certifying recognition algorithms for the considered graph classes. We focus on
the simplest case of threshold graphs. Our algorithm is the obvious one that
comes to mind, i.e., it removes all false twins of an input graph, and test the
resulting graph for membership in $\mathcal{F}$.   We focus on an efficient
implementation of this strategy which takes advantage of the structure of
threshold graphs and consequently is more efficient than standard algorithms for
removing twins, cf. \Cref{sec:conclusions}.

Recall that for any linear ordering $\prec$ of the vertex set of a graph $G$, we
can order the adjacency lists of all the vertices in $G$ with respect to $\prec$
in time $O(|V|+|E|)$.   Also, by using counting sort we can order the vertices
of a graph non-increasingly with respect to their degrees in time $O(|V|)$. So,
it takes time $O(|V|+|E|)$ to obtain the vertices of an input graph $G$ ordered
non-increasingly with respect to their degrees and having each of their
adjacency lists ordered accordingly.

Proceed by greedily colouring the vertices of $G$ using the aforementioned
ordering to obtain a colouring $c$.   For each degree $d_i$ in the degree
sequence of $G$, identify all the vertices of $G$ with degree $d_i$ and coloured
with the same colour under $c$; this is easily done by exploring the vertices of
degree $d_i$, and keeping a doubly linked list for each of the colours found so
far holding the vertices coloured with that colour. For each of these linked
lists, compare the neighbourhoods of all the vertices in the list: verify that
the first element of each list is the same; then verify that the second element
of each list is the same, and so on.   This can be easily achieved by using a
different pointer to traverse each list.   Since the explored lists have the
same length, if at some point two lists differ, then by exploring the rest of
these two lists it is possible to find a vertex which is in one of the lists and
not the other, and vice versa.   If the involved lists are $N(u)$ and $N(v)$,
and the procedure found that $x \in N(u) \setminus N(v)$ and $y \in N(v)
\setminus N(u)$, then $\{ u, v, x, y \}$ induces either a copy of $2K_2$, if $xy
\notin E(G)$, or a copy of $P_4$, if not.   In such a case, $\{ u, v, x, y \}$
is returned as a no-certificate.   In the case that all these tests pass, then
all the vertices with the same degree and coloured with the same colour are
false twins.

If $u$ and $v$ are twins in $G$, then they have the same degree, and they cannot
receive different colours because in a greedy colouring, the first colour
available for the first of them that was processed is also available for the
second.   Therefore, two vertices in $G$ are twins if and only if they have the
same degree and the same colour.   Therefore, by taking one vertex in each
colour class we obtain a point-determining graph $G'$.   It is now possible to
run any certifying recognition algorithm for threshold graphs running in time
$O(|V|+|E|)$ (e.g., the one presented in \cite{heggernesNJC14}); if such an
algorithm returns an induced copy of $2K_2$ or $P_4$, our algorithm also returns
this copy as a no-certificate.   If an induced copy of $C_4$ is found on $G'$,
say $(a, b, c, d, a)$, then the algorithm compares $N(a)$ to $N(c)$ and $N(b)$
to $N(d)$ to find, without loss of generality, a vertex $x \in N(a) \setminus
N(c)$ and a vertex $y \in N(b) \setminus N(d)$.   If $x$ is not adjacent to $b$
and $d$ in $G'$, then $(x,a,b,c)$ or $(x,a,d,c)$ is an induced $P_4$ in $G'$,
respectively. Analogously, if $y$ is not adjacent to $a$ and $c$, an induced
copy of $P_4$ is found in $G'$.   Finally, if $x$ is adjacent to $y$, then $\{
a, b, c, d, x, y \}$ induces a copy of $\overline{2P_3}$ in $G'$, and if not,
then $(x,d,c,y)$ is an induced $P_4$ in $G'$.   In any case, it is possible to
return a minimal obstruction as a no-certificate.   When the algorithm
successfully shows that $G'$ is a threshold graph, it then returns a split
partition together with the nested neighbourhood ordering as certificate of
membership; our algorithm returns these together with the colouring generated in
the first step as a yes-certificate.

The discussion above proves the following result.

\begin{theorem}
    There is a certifying algorithm that runs in time $O(|V|+|E|)$ to recognize
    fully threshold-colourable graphs.
    \hfill $\square$
\end{theorem}

\section{Centipedes}
\label{sec:centipedes}

In this section we focus on full $H$-colourability for a single graph $H$.   In
\cite{guzmanDM347} the author obtained an elegant answer when $H$ is a path.
When addressing this question for other trees $H$, we found the nicest answer in
the case of centipedes.   We start by recalling the solution of
\cite{guzmanDM347} for paths.   The graph $E$ is obtained from the claw by
subdividing two of its edges; it can also be obtained from the leftmost graph in
\Cref{fig:lekkerkerker} by deleting one leaf.   For a positive integer $n$, the
sets $C(n), O(n)$ and $LF(n)$ are defined as follows:
\[
    C(n) = \{ C_m \colon\ m = 3 \lor 5 \le m \le n+1 \},
\]
\[
    O(n) =
    \left\{ \begin{array}{cl}
        \varnothing   & \textnormal{if } n \le 4 \\
        \{ B \}       & \textnormal{if } n = 5   \\
        \{ A, B \}    & \textnormal{if } n = 6   \\
        \{ A, B, E \} & \textnormal{if } n \ge 7,
    \end{array} \right.
\]
and $LF(n)$ is the union $LF_1(n) \cup LF_2(n) \cup LF_3(n)$ where
\begin{align*}
    LF_1(n) &= \{ m_2P_2 \colon\ 3m_2 = n+2 \}, \\
    LF_2(n) &= \{ P_1 + m_2P_2 + m_4P_4 \colon\ 3m_2 + 5m_4 = n+1 \}, \\
    LF_3(n) &= \{ P_1 + m_2P_2 + m_4P_4 + m_6P_6 \colon\ 3m_2 + 5m_4 + 7m_6= n
    \}.
\end{align*}

\begin{theorem}\cite{guzmanDM347}
\label{thm:path-obs}
    Let $n$ be any positive integer. The set of minimal $P_n$-obstructions is
    $C(n) \cup LF(n) \cup O(n)$.
\end{theorem}

The most notorious characteristic of \Cref{thm:path-obs} is the fact that every
forest minimal obstruction to fully $P_k$-colourability is either the $E$ graph
or a linear forest having as connected components only copies of $P_1, P_2, P_4$
or $P_6$.   Although not as nice, centipedes share a similar feature, as the
only forest minimal obstructions to fully $H$-colourability when $H$ is a
centipede are the graph $T_2$ (the leftmost graph in \Cref{fig:lekkerkerker})
and linear forests.   In the following results, we use the term
\textit{$H$-obstruction} to refer to an obstruction to full $H$-colouring.

\begin{proposition}
\label{pro:centi-linear}
    Let $k$ be a positive integer, and let $H$ be a $k$-centipede.   If $k < 5$,
    then every forest minimal $H$-obstruction is a linear forest.   If $k \ge
    5$, then $T_2$ is the only forest minimal $H$-obstruction which is not a
    linear forest.
\end{proposition}

\begin{proof}
    Clearly $T_2$ is an $H$-obstruction.   We have that $3K_2$ is a minimal
    $H$-obstruction when $k \in \{3, 4\}$, also $2K_2$ is a minimal
    $H$-obstruction when $k = 2$, and finally $K_1 + K_2$ is a minimal
    $H$-obstruction when $k = 1$, so $T_2$ is not minimal in these cases.  It is
    trivial to verify the minimality of $T_2$ when $k \ge 5$.

    Let $G$ be a point-determining tree.   Since support vertices in a point
    determining tree have exactly one attached leaf, if $G$ has a vertex of
    degree at least $4$, then it contains $T_2$ as an induced subgraph, and if
    $G$ has a vertex of degree $3$, then $G$ either contains $T_2$ as an induced
    subgraph, or the vertex of degree $3$ is a support vertex which is the
    central vertex in an induced copy of $E$ in $G$.
    
    Suppose that $G$ does not contain $T_2$ as an induced subgraph and $v$ is a
    leaf attached to a vertex of degree $3$ in $G$.  If $G$ is a connected
    component of a forest $F$ such that $F-v$ admits a full $H$-colouring, then
    such an $H$-colouring must map the support of $v$ to a vertex of degree $3$
    in $H$.   Hence, $F$ also admits a full $H$-colouring, so it is not a
    minimal $H$-obstruction.  We conclude that the only forest minimal
    $H$-obstruction with maximum degree greater than $2$, is $T_2$.
\end{proof}

In what remains of this section, we provide general characteristics of the
linear forests that are minimal $H$-obstructions when $H$ is a $k$-centipede.
Unlike the case of full path-colouring, the length of connected components of
minimal obstructions in this scenario is not restricted to a constant set of
possibilities.

Notice that any linear forest is an induced subgraph of a long enough centipede.
For a linear forest $L$ we define $\mu_\mathcal{C}(L)$ to be the least integer
$k$ such that the $k$-centipede contains $L$ as an induced subgraph.   We also
define $a_L$ to be the number of connected components of $L$ of order greater
than or equal to $4$.

\begin{lemma}
\label{lem:mu}
    Let $r$ and $s$ be positive integers.   For a linear forest $L$ such that $L
    = \sum_{s=1}^r P_{n_s}$, with $n_s \notin \{1,3\}$ for each $s \in
    \{1, \dots, r\}$, the following equalities hold
    \[
        \mu_\mathcal{C}(L) = |V_L| - a_L - 1 = (m-1) + \sum_{n_s > 2}(n_s - 2)
        + \sum_{n_s = 2}(n_s - 1).
    \]
\end{lemma}
\begin{proof}
    For the second equality, notice that $|V_L| = \sum_{n_s > 2} n_s + \sum_{n_s
    = 2} n_s$, and hence $|V_L| - a_L - 1 = (-1) + \sum_{n_s > 2} (n_s - 1) +
    \sum_{n_s = 2} n_s$.   By adding $m-m$ to the right hand side of the
    equation we obtain the desired equality.

    For the remaining equality, by considering the embedding where each path in
    $L$ maximizes the number of leaves used, and skips one vertex from the spine
    between components, it is clear that $L$ is an induced subgraph of the
    centipede with spine of length $(m-1) + \sum_{n_s > 2}(n_s - 2) + \sum_{n_s
    = 2}(n_s - 1)$.   Similarly, any embedding of $L$ in the
    $\mu_\mathcal{C}(L)$-centipede must use at least one spine vertex for each
    $K_2$, and $n-2$ spine vertices for each $P_n$ with $n \ge 4$.   Also, spine
    vertices used by different components of $L$ must be at distance at least
    $2$.   Therefore, $\mu_\mathcal{C}(L) \ge (m-1) + \sum_{n_s > 2}(n_s - 2) +
    \sum_{n_s = 2}(n_s - 1)$.
\end{proof}

\Cref{lem:mu} directly yields the following result.

\begin{proposition}
\label{pro:unionmu}
    If $L_1$ and $L_2$ are disjoint linear forests having neither $K_1$ nor
    $P_3$ as connected components, then
    \[
        \mu_\mathcal{C}(L_1 + L_2) = \mu_\mathcal{C}(L_1) + \mu_\mathcal{C}
        (L_2) + 1.
    \]
    \hfill $\square$
\end{proposition}

Let $L$ be a point-determining linear forest, and let $v$ be a vertex of $L$. If
$L$ is isomorphic to $K_2$, or if it is isomorphic to $P_4$ and $v$ is a support
vertex, then $\mu_\mathcal{C}(L-v) = \mu_\mathcal{C}(L)$.   It is not hard to
observe that in any other case, $\mu_\mathcal{C}(L-v) < \mu_\mathcal{C}(L)$.
When $L$ is disconnected, $\mu_\mathcal{C}(L-v) = \mu_\mathcal{C}(L) - 2$ unless
$L$ is isomorphic to $2K_2$, or $v$ is a leaf not in a $K_2$, or $v$ is a
support vertex in a $P_4$; in such cases, $\mu_\mathcal{C}(L-v) =
\mu_\mathcal{C}(L) - 1$.  When $L$ is a path of length at least $5$, the
deletion of a leaf or a support vertex reduces the value of $\mu_\mathcal{C}(L)$
precisely by $1$, and the deletion of any other vertex decreases the value of
$\mu_\mathcal{C}(L)$ by $2$.

\begin{lemma}
\label{lem:replace}
    Let $k$ be a positive integer, let $H$ be a $k$-centipede, and let $L$ be a
    point-determining linear forest without isolated vertices. If $L_1$ is an
    induced subgraph of $L$ consisting of a subset of connected component, and
    $L_2$ is a point-determining linear forest such that $\mu_\mathcal{C}(L_1) =
    \mu_\mathcal{C}(L_2)$, then
    \begin{enumerate}
        \item $L$ is an induced subgraph of $G$ if and only if $(L-L_1) + L_2$
            is an induced subgraph of $G$.

        \item If neither $L_1$ nor $L_2$ is isomorphic to $rK_2$ for some
            positive integer $r$, then $L$ is a minimal $H$-obstruction if and
            only if $(L-L_1) + L_2$ is a minimal $H$-obstruction.

        \item If $L$ is a minimal $H$-obstruction and $L_1$ is isomorphic to
            $rK_2$ for some positive integer $r$, but $L$ is not isomorphic to
            $sK_2$ for some positive integer $s$, then $(L-L_1) + L_2$ is a
            minimal $H$-obstruction.
            
    \end{enumerate}
\end{lemma}

\begin{proof}
    For the first item, observe that one can consider an embedding of $L$ in
    $G$, and simply interchange $L_1$ for $L_2$ in order to obtain an embedding
    of $(L-L_1) + L_2$, and vice versa.   This also follows directly from
    \cref{pro:unionmu}.

    For the second item, we know from the first one that $L$ is an
    $H$-obstruction if and only if $(L-L_1) + L_2$ is.   Suppose that $L$ is
    minimal, and let $v$ be a vertex of $L$ not in $V_{L_1}$.   Again,
    interchanging $L_1$ for $L_2$ in an embedding of $L-v$ in $G$, we obtain an
    embedding of $((L-L_1) + L_2)-v$ in $G$.  Now, because $L_1 \not \cong
    rK_2$, there is a vertex $v$ of $L_1$ such that $\mu_\mathcal{C}(L_1-v) =
    \mu_\mathcal{C}(L_1) - 1$.   Since $L_2$ is isomorphic to $P_4$ if and only
    if $L_1$ is isomorphic to $P_4$ too, in which case the result is trivial, we
    may assume that $L_2 \not \cong P_4$, and hence, for any vertex $u$ of $L_2$
    we have $\mu_\mathcal{C}(L_2 - u) \le \mu_\mathcal{C}(L_2) - 1 =
    \mu_\mathcal{C}(L_1) - 1$.   Therefore, we can once again interchange
    $L_1-v$ for $L_2-u$ in an embedding of $L-v$ in $G$ to obtain an embedding
    of $((L-L_1) + L_2) - u$ in $G$.   Hence, $(L-L_1) + L_2$ is a minimal
    $H$-obstruction.   The remaining implication is analogous.

    For the third item, and by the choice of $L$, there is a vertex $v$ in $V_L
    \setminus V_{L_1}$ such that $\mu_\mathcal{C}(L-v) = \mu_\mathcal{C}(L)-1$.
    Hence, for any vertex $u$ in $L_2$ such that $\mu_\mathcal{C}(L_2-u) <
    \mu_\mathcal{C}(L_2)$, it follows from \cref{pro:unionmu} that
    $\mu_\mathcal{C}(((L-L_1) + L_2)-u) \le \mu_\mathcal{C}(L-v)$.   So, unless
    $L_2$ is isomorphic to $P_4$ and $u$ is a support vertex, we already have
    the desired result.   In the aforementioned case, it suffices to notice that
    $\mu_\mathcal{C}(((L-L_1) + L_2)-u) = \mu_\mathcal{C}(((L-L_1) + L_2)) - 1$,
    because the isolated vertex in $L_2 - u$ does not use a spine vertex.
\end{proof}

In the final result of this section we identify special families of minimal
$H$-obstructions when $H$ is a $k$-centipede.   These families are shown to
reach the minimum and maximum orders possible, so they provide a well determined
range where other minimal obstructions can exist.

\begin{proposition}
\label{pro:centi-max-min}
    Let $k$ be a positive integer, $k \ge 3$, and let $H$ be a $k$-centipede.
    \begin{enumerate}
        \item $(\lceil \frac{k}{2} \rceil + 1) K_2$ is a linear forest minimal
            $H$-obstruction of minimum order.

        \item If $k = 3r$ for some integer $r$, then $rP_4 + K_2$ is a linear
            forest minimal $H$-obstruction of maximum order.

        \item If $k = 3r+1$ for some integer $r$, then $(r+1)P_4$ is a linear
            forest minimal $H$-obstruction of maximum order.

        \item If $k = 3r+2$ for some integer $r$, then $rP_4 + 2K_2$ is a linear
            forest minimal $H$-obstruction of maximum order.
    \end{enumerate}
\end{proposition}

\begin{proof}
    Clearly, all of the proposed linear forests are minimal $H$-obstructions.
    The only linear forest minimal $H$-obstruction which is a tree is $P_{k+3}$,
    and the only minimal $H$-obstruction having $K_1$ as a connected component
    is $P_{k+2} + K_1$. Any other minimal obstruction $L$ is either
    disconnected, or does not have trivial components.   We consider some cases
    for $L$.
    
    If $L$ has a connected component $T$ having $\mu_\mathcal{C}(T) = 3r + 2$
    for some integer $r$, then \cref{lem:replace} shows that $(L-T) + (r+1)P_4$
    is also a minimal obstruction.   As $T$ has $3(r+1) + 1$ vertices, and the
    order of $(r+1)P_4$ is $4r+4$, we have $(L-T) + (r+1)P_4$ has order larger
    than $L$.   If $L$ has connected components $T_1$ and $T_2$ with
    $\mu_\mathcal{C}(T) = 3r$ and $\mu_\mathcal{C}(T) = 3s + 1$, then
    \cref{pro:unionmu} implies that $\mu_\mathcal{C}(T_1 + T_2) = 3(r+s) + 2$,
    so from \cref{lem:replace} we obtain that $(L-(T_1+T_2)) + (r+s+1)P_4$ is
    also a minimal obstruction.   The orders of $T_1 + T_2$ and $(r+s+1)P_4$ are
    $3(r+s+1) + 2$ and $4(r+s+1)$, so clearly $(L-(T_1+T_2)) + (r+s+1)P_4$ has
    order larger than $L$.  If $L$ has connected components $T_1, T_2$ and $T_3$
    such that $\mu_\mathcal{C}(T_1) \equiv \mu_\mathcal{C}(T_2) \equiv
    \mu_\mathcal{C}(T_3)$ modulo $3$, then $\mu_\mathcal{C}(T_1 + T_2 + T_3)
    \equiv 2$ modulo $3$, so this case can be dealt similarly to the previous
    two cases to replace $T_1 + T_2 + T_3$ with a copy of $rP_4$ in $L$,
    resulting in a larger minimal $H$-obstruction.   If $L$ has a component
    $T_1$ such that $\mu_\mathcal{C}(T_1) = 3r+1$ then, using
    \cref{lem:replace}, it can be substituted by $rP_4 + K_2$ to obtain a larger
    minimal obstruction; an analogous argument applies when there are components
    $T_1$ and $T_2$ such that $\mu_\mathcal{C}(T_1 + T_2) \equiv 1$ modulo $3$.
    Finally, if there is a component $T_1$ having $\mu_\mathcal{C}(T_1) = 3r$,
    then \cref{lem:replace} again implies that $T_1$ can be replaced by
    $(r-1)P_4 + 2K_2$ to obtain a minimal $H$-obstruction with at least the same
    order as $L$.   Using the reductions described in this paragraph, for any
    forest minimal $H$-obstruction another minimal $H$-obstruction of larger (or
    equal) order can be obtained, such that each of its components is a $P_4$,
    except maybe for two of them, which are copies of $K_2$.

    An analogous argument shows that the minimal $H$-obstruction consisting of
    copies of $K_2$ has minimum order among minimal $H$-obstructions which are
    linear forests.
\end{proof}

Our results up to this point are sufficient to obtain a result similar to
Corollary 10 in \cite{guzmanDM347}.

\begin{corollary}
\label{cor:centi-fam}
    Let $\mathcal{C}$ be the family of all centipedes.   A graph $G$ admits a
    full $\mathcal{C}$-colouring if and only if it is free from every cycle of
    length different from $4$ and the graphs $A, B$ and $T_2$ (the leftmost
    graph in \Cref{fig:lekkerkerker}).
\end{corollary}
\begin{proof}
    If $G$ admits a full $\mathcal{C}$-colouring, then the result follows
    directly from \Cref{thm:C4-consequences,pro:centi-linear}.  Conversely, let
    $G$ be a graph free from $A, B, T_2$ and every cycle of length different
    from $4$.  When $H$ is a $k$-centipede the first item in
    \Cref{pro:centi-max-min} implies that the order of the linear forest minimal
    obstructions to full $H$-colouring grow with respect to $k$.   Hence, there
    is a large enough integer $\ell$ such that if $H$ is the $\ell$-centipede,
    then $G$ is free from every minimal obstruction to full $H$-colouring.
\end{proof}

Observe that if we apply \Cref{lem:C4} to the family of the forests in which
every connected component is a caterpillar, we obtain the same set of minimal
obstructions described in \Cref{cor:centi-fam}.   This points out that a graph
admits a full caterpillar-colouring if and only if it admits a full
centipede-colouring (which can also be easily proved directly).

\section{Future directions of research}
\label{sec:conclusions}

Let $\mathcal{F}$ be a hereditary class of graphs.  Recall that we use
$\mathcal{F}$-\textsc{Recognition} to denote the recognition problem for
$\mathcal{F}$.   We now state the problem presented at the beginning of
\Cref{sec:main}.

\begin{problem}
\label{prb:npc-poly}
    Is there a hereditary family $\mathcal{F}$ such that
    $\mathcal{F}$-\textsc{Recognition} is NP-complete and \fgc{\mathcal{F}} is
    polynomial-time solvable?
\end{problem}

We can reduce the search space when looking for candidates to obtain a positive
answer for \Cref{prb:npc-poly}.   Suppose that $M$ is a polynomial time
algorithm to solve \fgc{\mathcal{F}}.   Consider an input graph $G$ for
$\mathcal{F}$-\textsc{Recognition} and run $M$ on $G$.   If $M$ answers
\texttt{false}, then $G$ is not in $\mathcal{F}$, so a possible algorithm to
solve $\mathcal{F}$-\textsc{Recognition} can safely answer \texttt{false}.   If
$M$ answers \texttt{true}, then it could still be the case that $G$ does not
belong to $\mathcal{F}$, but this could only happen if $G$ contains a graph in
$\obs(\mathcal{F})$ which is not point-determining.   Therefore, if there are
finitely many elements of $\obs(\mathcal{F})$ which are not point-determining,
we could brute-force test whether any of them is an induced subgraph of $G$. The
procedure we just described is a polynomial-time reduction from
$\mathcal{F}$-recognition to \fgc{\mathcal{F}} when $\obs(\mathcal{F})$ has
finitely many elements which are not point-determining.   Therefore, a positive
answer to \Cref{prb:npc-poly} must involve a hereditary class having infinitely
many not point-determining minimal obstructions.

In the algorithmic setting, in \Cref{sec:algo} we presented a certifying
algorithm to recognize fully threshold-colourable graphs which runs in
$O(|V|+|E|)$ time and uses only elementary data structures.   Linear algorithms
could also be obtained using standard approaches like modular decomposition or a
clever use of multiple passes of sorting algorithms, nonetheless, our algorithm
is far simpler to implement and it is optimized by stopping when a
no-certificate is found.   We accomplished that by taking advantage of the
structure of threshold graphs.   In this context, we propose the following
problem.

\begin{problem}
\label{prb:algorithms}
    Find simple certifying algorithms to recognize fully
    $\mathcal{F}$-colourable graphs for the other families $\mathcal{F}$
    considered in \Cref{thm:C4-consequences}.
\end{problem}

Recall that there is an algorithm to recognize fully $\mathcal{F}$-colourable
graphs running in the same time bound as an algorithm to decide membership in
$\mathcal{F}$.   Thus, an algorithm that positively answers
\Cref{prb:algorithms} is expected to have the same running time as recognizing
the family $\mathcal{F}$.

Lastly, we highlight the problem we discussed in \Cref{sec:main} regarding the
family of minimal obstructions to full $\mathcal{F}$-colouring when
$\mathcal{F}$ is a hereditary family of forests. In order to have a better
understanding of such a family, we proposed the problem of finding, for any
positive integer $k$, every point determining forest having independence number
equal to $k$.   It would be good to know if there is a nice answer to this
question.


\section*{Acknowledgements}
We kindly thank Seyyed~Aliasghar~Hosseini for his involvement in a preliminar
version of this work and the useful discussions we had on the subject.

We are grateful to Ross McConnell for pointing us to a linear-time algorithm for
finding all the false twin vertices in a graph.   We are also grateful to
Santiago Guzm\'an-Pro for many useful discussions that led to a better
presentation of the material in this work and to proposing \Cref{prb:npc-poly}.


\end{document}